\documentclass[12pt,a4paper,english]{amsart}
\usepackage[english,french]{babel}
\usepackage{amssymb,xspace}
\usepackage{amstext}
\usepackage{smfthm}
\theoremstyle{plain}
\usepackage{amsbsy,amssymb,amsfonts,latexsym}
\usepackage[dvips]{graphicx}
\usepackage{epsfig}
\usepackage[all]{xy}
\newcommand{\ds }{\ensuremath{\displaystyle}}
\newcommand{\R }{\ensuremath{\mathbb R}}
\newcommand{\C }{\ensuremath{\mathbb C}}

\newcommand{\N }{\ensuremath{\mathbb N}}
\renewcommand{\P }{\ensuremath{\mathbb P}}

\font\calli=callig15
\def\gg{\hbox{\!\!\calli G}}
\def\ii{\mathcal{I}}

\newtheorem{theorem}{Theorem}[section]

\newtheorem{lemma}[theorem]{Lemma}

\newtheorem{proposition}[theorem]{Proposition}

{\theoremstyle{definition}}

{\theoremstyle{definition}}

{\theoremstyle{definition}}

{\theoremstyle{definition}}

{\theoremstyle{definition}}

\newenvironment*{absen}
{
\begin{center}
\begin{minipage}{13 cm}
\selectlanguage{english}
\footnotesize
\setlength{\baselineskip}{0.5\baselineskip}
\textsc{Abstract. ---}}
{\setlength{\baselineskip}{2\baselineskip}
\normalsize
\end{minipage}
\end{center}}

\newenvironment*{absfr}
{
\begin{center}
\begin{minipage}{13 cm}
\selectlanguage{french}
\footnotesize
\setlength{\baselineskip}{.5\baselineskip}
\textsc{R\' esum\' e. ---}}
{\setlength{\baselineskip}{2\baselineskip}
\normalsize
\end{minipage}
\end{center}}
\email{julien.grivaux@free.fr\ }

\title{Tian's invariant of the Grassmann Manifold}
\author{Julien Grivaux}
\address{
Universit\'{e} Pierre et Marie Curie}

\newcommand{\gm }{Grassmann manifold}
\newcommand{\gpqc }{\ensuremath{G_{p,q}(\C \, )}}
\newcommand{\prc }[1]{\ensuremath{\P ^{#1}(\C)}}
\newcommand{\prr }[1]{\ensuremath{\P ^{#1}(\R)}}

\newcommand{\mpq }{\ensuremath{M^{*}(p+q,p)}}
\newcommand{\upq }{\ensuremath{U(p+q)}}
\newcommand{\gpq }{\ensuremath{\gg_{p,q}}}
\newcommand{\dd }{\ensuremath{\partial\overline{\partial}}}
\newcommand{\ba }[1]{\ensuremath{\overline{\!#1}}}
\newcommand{\ic }{\ensuremath{I^{c}}}
\newcommand{\iti }{\ensuremath{\widetilde{\!I}}}
\newcommand{\itic }{\ensuremath{\tilde{I}^{c}}}
\newcommand{\hc }{\ensuremath{\breve{H}}}
\newcommand{\gde }[3]{\ensuremath{E_{#1,#2}^{#3}}}
\newcommand{\gdec }[3]{\ensuremath{\breve{E}_{#1,#2}^{#3}}}
\newcommand{\ma }[2]{\ensuremath{M_{#1,#2}(\C\, )}}

\newcommand{\mum }{\ensuremath{\mu _{_M}}}
\newcommand{\pei }{\ensuremath{P_{I}}}
\newcommand{\mm }[1]{\ensuremath{m_{{#1}}}}
\newcommand{\lga }{\ensuremath{\rightarrow}}

\newcommand{\ti }{\ensuremath{\tilde}}
\newcommand{\RR}{\ensuremath{\mathcal{R}}}
\newcommand{\tr}{\ensuremath{Tr}}
\newcommand{\cinf }{\ensuremath{C^{\infty }}}

\DeclareMathOperator{\id}{Id}

\begin{document}
\maketitle
\begin{absen}
We prove that Tian's invariant on the complex Grassmann
manifold $\gpqc$ is equal to $1/(p+q)$. The method
introduced here uses a Lie group of holomorphic isometries
which operates transitively on the considered manifolds and
a natural imbedding of $\bigl(\P^{1}(\C\, )\bigr)^{p}$ in
$\gpqc$.
\par
\end{absen}
\par\bigskip
\begin{absfr}
On prouve que l'invariant de Tian sur la grassmannienne $\gpqc$
est $1/(p+q)$. La m\'{e}thode pr\'{e}sent\'{e}e dans cet
article utilise un groupe de Lie d'isom\'{e}tries
holomorphes qui op\`{e}re transitivement sur les
vari\'{e}t\'{e}s consid\'{e}r\'{e}es ainsi qu'un plongement
naturel de $\bigl(\P^{1}(\C\, )\bigr)^{p}$ dans $\gpqc$.
\par
\end{absfr}
\selectlanguage{english}
\section{Introduction}
On a complex manifold, an hermitian metric $h$ is characterized  by the
\textit{1-1 symplectic} form $\omega $ defined
by $\omega =i\, g_{\lambda \, \ba{\mu }}
\, dz^{\lambda }\wedge d\, \ba{z}^{\mu }$, where $g_{\lambda \, \ba{\mu }}=
h_{\lambda \, \ba{\mu }}/2$.
\par
The metric is a \textit{K\"{a}hler metric}
if $\omega $ is closed, i. e. $d\omega =0$; then $M$ is a
\textit{K\"{a}hler manifold}.
\par
On a K\"{a}hler manifold, we can define
the \textit{Ricci form} by
$R =i\, R_{\lambda \, \ba{\mu }}
\, dz^{\lambda }\wedge d\, \ba{z}^{\mu }$, where
$R_{\lambda \, \ba{\mu }}=-\partial _{\lambda \, \ba{\mu  }}\log |g|$.
\par
A K\"{a}hler manifold is \textit{Einstein with factor $k$} if
$R=k\omega $. For instance,
choosing a local coordinate system $Z=(z_{1},\dots
,z_{m})$,
the projective space $\P_{m}(\C\, )$ with
the Fubini-Study metric $\omega =i\partial\,  \ba{\partial }\log \bigl(
1+||Z||^{2}\bigr)$ is Einstein with
factor $m+1$.
\par
On a K\"{a}hler manifold $M$, the \textit{first Chern class} $C^{1}(M)$
is the cohomology class of the Ricci tensor, that is the set of the
forms $R+i\partial \, \ba{\partial }\varphi $, where $\varphi $ is $\cinf$
on $M$. If there is a form in $C^{1}(M)$ which is positive (resp. negative,
zero), then $C^{1}(M)$ is \textit{positive} (resp. \textit{negative},
\textit{zero}).
If a K\"{a}hler manifold is Einstein, then $C^{1}(M)$ and $k$ are both
positive (resp. negative, zero). In the negative case, it was
proved by Aubin (\cite{Aub1}, see also \cite{Aub8}),
that there exists a unique Einstein-K\"{a}hler metric
(E.K. metric) on $M$ .
It is so for the zero case too (\cite{Aub1}, \cite{Ya}).
The question for the positive case is still
open: some manifolds, such as the complex projective space blown up at
one point, do not admit an  E.K. metric
(for obstructions, see \cite{Li} and \cite{Fu}). Aubin
\cite{Aub7} and Tian \cite{Ti} have shown that for
suitable values of holomorphic invariants of the metric, there exists an
E.K. metric on $M$.
\par
For $\omega /2\pi $ in $C^{1}(M)$, \textit{Tian's invariant}
$\alpha (M)$
is the supremum of
the set of the real numbers $\alpha $ satisfying the following:
there exists a constant $C$
such that the inequality $\ds\int_{M}{e^{-\alpha \varphi }}\leq C$
holds for all
the $\cinf$ functions $\varphi $ with $\omega
+i\partial \, \ba{\partial }\varphi >0$
and $\sup\varphi \geq 0$,
where $\omega =
i\, g_{\lambda \, \ba{\mu }}\, dz^{\lambda }\!\wedge d\, \ba{ z}^{\, \mu }$
is the metric form.
Such functions $\varphi $ are said
\textit{$\omega $-admissible}.
\par
In \cite{Ti},
Tian established  that if $\alpha (M)>m/(m+1)$, $m$
being the dimension of
$M$, there exists an E.K. metric\ on $M$. This condition is not
necessary: it does not hold on the projective space, where Tian's
invariant is $1/(m+1)$.
\par
In the same paper, Tian introduces a more restrictive
invariant $\alpha
_{G}(M)$,
considering only
the admissible functions $\varphi $ invariant by the
action of a compact group $G$ of holomorphic isometries. The sufficient
condition for the existence of an E.K. metric on $M$ remains
$\alpha _{G}(M)>m/(m+1)$; it is more easily satisfied if the group $G$
is rich enough.
\par
In many cases, the group $G$ is a non-discrete Lie group. The invariant
$\alpha _{G}(M)$ can be computed using subharmonic functions methods and
the maximum principle (for effective examples, see \cite{Be1},
\cite{Be2}, \cite{BeCh1}, \cite{BeCh2}, \cite{Re}).
\par
In this paper, we prove the following theorem:
\begin{theorem}\label{thIntroDeux}
Tian's invariant on $\gpqc$ is given by
$\alpha \left(\gpqc\right)=1/(p+q) $.
\end{theorem}
This generalizes the known result on $\P^{m}(\C\,
)$ (\cite{Ti}, see also \cite{Aub6}). Let us also mention that
Tian's invariant has been computed on $\P^{m}(\C\,
)$ blown up at one point and on certain Fermat
hypersurfaces using H\"ormander $L^{2}$ estimates for
the $\overline{\partial }$-equation (\cite{Ti}).
\par
We first compute the volume element of the metric $\gpq$; then
we will establish some general preliminary results concerning Tian's
invariant as well as imbeddings of $\bigl\{\P^{1}(\C\,)\bigr\}^{p}$ in $\gpqc$ which allow
us to deduce $\alpha \left(\gpqc\right)$ from
$\alpha \left(\P^{1}(\C\,)\right)$.
\section{Basic properties of the Grassmann manifold}
We propose here a short survey of the properties of \gm\ (for more details,
see \cite{KoNo}).
We denote by $\gpqc$ the set of the subspaces of dimension $p$
in $\C^{\, p+q}$; in particular,
$G_{1,m}(\C \, )$ is the complex projective space of dimension $m$. It is
known
(see \cite{Aub6})
that on $\P_{m}(\C\, )$, the Fubini-Study metric is Einstein with factor $m+1$
and that Tian's invariant is $1/(m+1)$. Now, let $\mpq$ be the
set of the matrices of rank $p$ in $\ma{p+q}{p}$.
The group $Gl_{p}(\C\, )$ acts by
multiplication on the right on $\mpq$.
 More precisely
$\bigl(\mpq,\ \pi,\ \gpqc\bigr)$ is a principal fiber bundle with group
$Gl_{p}(\C\, )$. The group $Gl_{p+q}(\C\, )$ acts by multiplication
on the left on
$\mpq$ and induces an action on
$\gpqc$;
so does the unitary group $\upq$. These groups act transitively on
$\gpqc$, which shows that $\gpqc$ is compact.
\par
We denote by $\ii$ the set of all increasing-ordered subsets of
$p$ elements in
$\{1,\dots ,p+q\}$.
Let $P$ be an element of $\mpq$, $P=\bigl(p_{ij}\bigr)_{1\leq i\leq p+q
\atop 1\leq j\leq p}$.
By Cauchy-Binet formula we get:
$\det \bigl(\, ^{t}\!P\, \ba{P}\, \bigr)=
\sum _{I\in\ii}|\det m_{_I}(P)|^{2}$,
where $m_{I}(P)$ is the matrix
$\bigl(p_{ij }\bigr)_{i\in  I \atop 1\leq j\leq p}$.
The form $\omega $,
where
$\omega =i\, \dd \, \log \det \bigl(\, ^{t}\!P\, \ba{P}\,
\bigr)$,
is invariant by the action of $Gl_{p}
(\C\, )$ on $\mpq$, and so it projects onto a form $\gpq$.
The metric
$\gpq$ is a K\"{a}hler metric form on $\gpqc$.
For $p=1$,
this metric on $G_{1,m}(\C\, )$ is the Fubini-Study metric on the complex
projective space. The action of the unitary group $\upq$ on $\gpqc$
preserves the metric $\gpq$ so that $\upq$ is a group of
holomorphic isometries which operates transitively on $\gpqc$.
\par
For $I$ in $\mathcal{I}$, let
$U_{I}$ be the set of the matrices $P$ in $\mpq$ such that
$\det(m_{I}(P))$ is non-zero. Then
$\pi (U_{I})$ is
a coordinate open set on
$\gpqc$, the matrix $Z_{I}$ in $\ma{q}{p}$ is the coordinate,
the inverse of the chart $\varphi _{I}$ sends $\mpq$ onto $\pi (U_{I})$
and we have
$m_{I}\bigl(\varphi_{I} ^{-1}(Z_{I}) \bigr)=I^{(p)}$ where
$I^{(p)}$ is the $p\times p$ identity matrix, and
$m_{\ic}\bigl(\varphi _{I}^{-1}(Z_{I}) \bigr)=Z_{I}$.

\begin{lemma}\label{SecUnLemmeUn}
For $I$ in $\ii$, let $\lambda _{I}$ be the map
from $\pi (U_{I})$ to  $\R_{+}$
defined by $$\lambda _{I}(Z_{I})=\bigl|\, \det(\id+
^{t}Z_{I}\, \ba{Z}_{I})\,  \bigr|^{-(p+q)}.$$ Then $\bigl(\lambda _{I} \bigr)
_{I\in\ii}$ are the components of a maximal differential form
$\eta $ on $\gpqc$, namely:
$$\eta =\lambda _{I}\bigl(i/2 \bigr)^{pq}\,
\bigl(dZ
\!\wedge d\, \ba{Z}\bigr)_{I}.$$
\end{lemma}
\begin{proof}
It suffices to show that the following transformation rule holds:
$$
\mbox{for every}\ I,\ \iti\ \mbox{in}\  \mathcal{I},\  \lambda _{I}\
\mbox{is equal to}\
\lambda _{\, \iti}\times\left|\, \det\, \,
\dfrac {\partial\, Z_{\, \iti}}{\partial\, Z_{I}}\,
\right|^{2}\ \mbox{on}\ \pi (U_{I})\cap \pi (U_{\, \iti}).
$$
Let $\pei$
be the matrix
$\varphi ^{-1}_{I}(Z_{I})$.
Then $\pei\bigl\{\mm{\, \iti}(\pei) \bigr
\}^{-1}=P_{\, \iti}$,
so $Z_{\, \iti}=\mm{\, \itic}(\pei)\, \bigl\{\mm{\, \iti}
(\pei) \bigr
\}^{-1}$.
The differential of the map
which sends $Z_{I}$ on $\pei$
is the map which sends $H$ on $\hc$, where
$\mm{\, \ic}(\hc)=H$ and $\mm{I}(\hc)=0$.
The change of charts sending $Z_{I}$ on $Z_{\, \iti}$, we obtain
\begin{align*}
D\, Z_{\, \iti}(H)&=\mm{\, \itic}(\hc)\, \bigl\{\mm{\, \iti}(\pei) \bigr\}^{-1}
-\mm{\, \itic}(\pei)\, \bigl\{\mm{\, \iti}(\pei) \bigr\}^{-1}\,
\mm{\, \iti}(\hc)\, \bigl\{\mm{\, \iti}(\pei) \bigr\}^{-1}\\
&=\bigl(\mm{\, \itic}(\hc)-\gamma \, \mm{\, \iti}(\hc) \bigr)\,
\alpha ^{-1},\\
\mbox{where}\quad
\alpha &=\mm{\, \iti}(\pei),\ \beta
=\mm{\, \itic}(\pei)\ \mbox{and}\ \gamma =\beta \, \alpha ^{-1}.
\end{align*}
Let us define
a map $u$ from $\ma{q}{p}$ to $\ma{q}{p}$
by $u(H)=\mm{\, \itic}(\hc)-
\gamma \, \mm{\, \iti}(\hc)$.
We can choose $I=\{q+1,\dots ,q+p\}$ and
$\iti =\{1,\dots ,r\}\cup\{q+1+r,\dots ,q+p\}$, where $0\leq r\leq \inf
(p,q)$.
We define the $k\times l $ matrix $\gde{i}{j}{(k\times l)}$ by
$\bigl(\gde{i}{j}{(k\times l)} \bigr)_{\lambda \mu }=
\delta _{i\lambda }\, \delta _{j\mu }$.
We have
\begin{align*}
\mm{\, \iti}
\bigl(\gdec{i}{j}{(q\times p)} \bigr)&=\gde{i}{j}{(p\times p)}
\quad  \mbox{if}\ i\leq
r,\  \mbox{and}\ 0\ \mbox{if}\  i>r,\\
\mbox{and}\quad
\mm{\, \itic}
\bigl(\gdec{i}{j}{(q\times p)} \bigr)&=\gde{i-r}{j}{(q\times p)}
\quad  \mbox{if}\ i>r,\ \mbox{and}\ 0\ \mbox{if}\ i\leq r.\ \mbox{Hence}\\
\quad
\bigl(\gamma \, \mm{\, \iti}(\gdec{i}{j}
{(q\times p)} ) \bigr)_{\alpha \beta }&=
\gamma _{\alpha i}\, \mm{\, \iti}(\gdec{i}{j}{(q\times p)})_{i j}\,
\delta _{j\beta }=\gamma _{\alpha i}\, \delta _{j\beta }
\quad \mbox{if}\
i\leq r,\ \mbox{and}\ 0\ \mbox{elsewhere}.
\end{align*}
Now the map
which sends $H$ to $\gamma \,\mm{\, \iti}(\hc)$
can be restricted if $1\leq j\leq p$
to the span $B_{j}$ of the $\bigl( E_{i, j }\bigr)
_{1\leq i\leq q}$.
The $r$ first columns of its matrix
are those of $\gamma $, the others are $0$.
The map which sends $H$ to
$\gamma \, \mm{\, \itic}(\hc)$ maps also $B_{j}$ into itself.
The
right upper block of its matrix is $I^{(q-r)}$, the other elements are
$0$.
This allows us to compute the matrix of the restriction of $u$ to
$B_{j}$, whose determinant is $(-1)^{r\times (q-r)}\det
\bigl(\gamma _{i j} \bigr)_{q-r+1\leq i\leq q\atop 1\leq j\leq r}$.
So
$\det u =(-1)^{p\times r\times (q-r)}\,\Bigl[
\det
\bigl(\gamma _{i\, j} \bigr)_{q-r+1\leq i\leq q\atop 1\leq j\leq
r}\Bigr]^{p}$.
For $1\leq i\leq q$, let $C_{i}$ be the span of the
$(E_{i,j})_{1\leq j\leq p}$.
Each $C_{i}$ is stable by the
map
from $\ma{q}{p}$ to $\ma{q}{p}$ which sends $H$ to $H\, \alpha ^{-1}$.
The matrix of the restriction is $\alpha
^{-1}$, so the determinant of the map is $(\det \alpha )^{-q}$. Hence
$$
\bigl|\, \det  DZ_{\, \iti}(H)\, \bigr|^{2}=\bigl|\, \det\, (\gamma _{i,j})
_{{}_{q-r+1\leq i\leq q\atop 1\leq j\leq r}}\,  \bigr|^{2p}\times \bigl|
\, \det\alpha
 \, \bigr|^{-2q}.
$$
Let $A$ be the right $r\times r$ upper block of
$\alpha $. The left $(p-r)\times (p-r)$ lower block of $\alpha $ is
$I^{(p-r)}$ and the right $(p-r)\times r$ lower block is $0$, so
$\det\alpha =(-1)^{r(p-r)}\det A$. The left $r\times (p-r)$ lower block
of $\beta $ is $0$, the right $r\times r$ block is $I^{(r)}$ so that the
left $r\times r$ lower block of $\gamma $ is $A^{-1}$.

From this we deduce $\bigl|\det\, DZ_{\, \iti}(H) \bigr|^{2}=\bigl|\,
\det\alpha  \, \bigr|^{-2(p+q)}$. Since $P_{I}\, \alpha ^{-1}=P_{\, \iti}$, we
have
$$
\lambda _{\, \iti}=\Bigl|\, \det\bigl(\, ^{t}\!P_{\, \iti}\,
\ba{P}_{\, \iti}\, \bigr) \, \Bigr|^{-(p+q)}
=\bigl|\, \det\alpha \,  \bigr|^{2(p+q)}\lambda _{I}=
\left|\, \det\, \dfrac {\partial Z_{\, \iti}}
{\partial Z_{I}}\, \right|^{-2} \lambda _{I}.
$$
\end{proof}
\begin{lemma}\label{SecUnLemmeDeux}
The unitary group $\upq$ preserves $\eta $.
\end{lemma}
\begin{proof}
We call $I$ the set $\{q+1,\dots ,q+p\}$. We define $P_{I}$
in $\pi (U_{I})$ by $P_{I}=\varphi ^{-1}_{I}(Z_{I})
$. Let $U$ be an element in $\upq$ such that  $\mm{I}(UP_{I})$ is
invertible. Let $\tilde{P}_{I}=UP_{I}\, \bigl\{\mm{I}(UP_{I}) \bigr\}^{-1}$
and $\tilde{Z}_{I}=\mm{\ic}(\tilde{P}_{I})$.
We have
$\ti{Z}_{I}=\mm{\ic}(U)\, P_{I}\, \bigl\{\mm{I}(U)P_{I}
\bigr\}^{-1}$.
So
$$
D\ti{Z}_{I}(H)=\mm{\ic}(U)\, \Bigl[
\hc\bigl\{
\mm{I}(U)P_{I}
\bigr\}^{-1}
-P_{I}\bigl\{
\mm{I}(U)P_{I}
\bigr\}^{-1}\mm{I}(U)\, \hc\, \bigl\{
\mm{I}(U)P_{I}
\bigr\}^{-1}
\Bigr]
.$$
Thus $D\ti{Z}_{I}(H)=
X \hc \delta ^{-1}$, where $\delta =\mm{I}(U)P_{I}$ and
$X=\mm{\ic}(U)\bigl[I^{(p+q)}-P_{I}\, \delta ^{-1}\, \mm{I} (U)\bigr]$.
Let $X_{1}$ be the $q\times q$ matrix of the $q$ first columns of $X$.
Then, $X\hc=X_{1}H$ and we get $D\ti{Z}_{I}(H)=X_{1} H \delta
^{-1}$. The determinant of the map from $\ma{q}{p}$ to $\ma{q}{p}$ which
sends $H$ to $H \delta  ^{-1}$
is
$(\det\delta )^{-q}$. The determinant of the map from
$\ma{q}{p}$ to $\ma{q}{p}$ which sends $H$ to $X_{1} H$
is $(\det X_{1})^{p}$, so
$\det D\ti{Z}_{I}=(\det X_{1})^{p}\, (\det \delta )^{-q}$.
We divide $U$ into four blocks:
$$U=\left(
\begin{array}{cc}
U_{q}&U_{q,p}\\
U_{p,q}&U_{p}
\end{array}
\right),\quad U_{q}\in M_{q}(\C\, ),\ U_{p}\in M_{p}(\C\,),\
U_{p,q}\in M_{p,q}(\C\, ),\ U_{q,p}\in M_{q,p}(\C\, ).
$$
Then
$\delta =U_{p,q}\, Z_{I}+U_{p}$, so
$X_{1}=U_{q}-\bigl(U_{q}\, Z_{I}+U_{q,p} \bigr)\,
\bigl(U_{p,q}\, Z_{I}  +U_{p}\bigr)^{-1}\, U_{q,p}$. Let
$Z$ in $\ma{p+q}{p+q }$ be the matrix with blocks $Z_{q}=I^{(q)}$, $Z_{p,q}=0$,
$Z_{q,p}=Z_{I}$, $Z_{p}=I^{(p)}$,
the notations being the same as above.
Writing $\det U=\det (U Z)$ and
using the column transformation
$C_{1}\leftarrow C_{1}-C_{2}\, \bigl(U_{p,q}\, Z_{I} +
U_{p} \bigr)^{-1}\, U_{p,q}$ where $C_{1}$ is made of the first
$q$ columns and $C_{2}$ of the remaining ones, we get
$$
\det U=\det\Bigl[
U_{q}-\bigl(
U_{q}\, Z_{I}+U_{q,p}
\bigr)\,
\bigl(
U_{p,q}\, Z_{I}+U_{p}
\bigr)^{-1}\, U_{p,q}
\Bigr]\times \det \bigl(U_{p,q}Z_{I}+U_{p} \bigr).
$$
Hence $\bigl|\det \, D\ti{Z}_{I} \bigr|^{2}=\bigl|\det \delta \bigr|^{-2(p+q)}$.
We have $\ti{P}_{I}=A P_{I} \delta ^{-1}$, so
$$
\lambda _{\, \iti}=\det \bigl(\, ^{t}\!\ti{P}_{I}\, \ba{\ti{P}}_{I} \bigr)
^{-(p+q)}=\det \bigl(\, ^{t}\!P_{I}\, \ba{P}_{I} \bigr)^{-(p+q)}\times
\bigl|\det \delta  \bigr|^{\, 2(p+q)}=\lambda _{I}\,
\bigl|\det D\ti{Z}_{I} \bigr|^{-2},
$$
which proves the result.
\end{proof}
\begin{proposition}\label{SecUnTheoUn}
\begin{enumerate}
  \item [1.]
$dV\left(\gpq\right)=\eta $.
  \item [2.]
If $I\in\mathcal{I}$, $\left|\gpq\right|_{I}=
\Bigl\{\det\bigl(I^{(p)}+^{t}\!Z_{I}\, \ba{Z}_{I} \bigr)  \Bigr
\}^{-(p+q)}$.
  \item [3.]
  $\RR\left(\gpq\right)=(p+q)\gpq$.
\end{enumerate}
\end{proposition}
\begin{proof}
$1$. Let $I$ in $\mathcal{I}$. It is easy to compute $\gpq$ at
the point $Z_{I}=0$:
$\gpq (H,K)=\tr (H\, \ba{K})$. Then $dV(\gpq)_{\bigm|Z_{I}=0}=\bigl(i/2 \bigr)
^{pq}\, \bigl(dZ\wedge d\, \ba{Z}\bigr)_{I}=\eta _{\bigm|Z_{I}=0}$. Since
$dV(\gpq)$ and $\eta $ are invariant by the transitive action of $\upq$,
we have $dV(\gpq)=\eta $.
\par\medskip
$2$. Since $dV(\gpq)=\bigl|\gpq \bigr|_{I}\, (i/2)^{pq}\,
\bigl(dZ\wedge d\, \ba{Z}\bigr)_{I}$, property
$1$ gives the result.
\par\medskip
$3$. Remark that $\gpq=i\, \dd\, \log \bigl\{\det(I^{(p)}+
^{t}\!\!Z_{I}\, \ba{Z}_{I}) \bigr\}$. Since $\RR \left(\gpq\right)
=-i\, \dd\, \log \left|\gpq\right|_{I}$, we obtain
$\RR\left(\gpq\right)=(p+q)\gpq$, which expresses that $\gpq$ is
Einstein, with factor $p+q$.
\end{proof}
\section{Some general results about Tian's invariant}
\subsection{Tian's invariant with a normalization on a finite set}
If $X$ is a manifold, we will denote by $\mu _{X}$ a measure on $X$
compatible with the manifold structure.
\begin{theorem}\label{SecDeuxUnProp}
Let $M$ be a compact K\"{a}hler manifold. We suppose that
there exists
a compact Lie group $G$ of holomorphic
isometries. Let $\Delta _{n}=\{P_{1},\dots ,P_{n}\}$
be a finite subset of
$M$. Let $\alpha (\omega )$ (resp. $\alpha _{_{\Delta _{n}}}(\omega )$ )
be the supremum of the set of the nonnegative real numbers $\alpha $
satisfying the condition: there exists a constant $C$ such that the
inequality $\ds\int _{M}e^{-\alpha \varphi }\leq C$ holds for all the
$\omega $-admissible functions $\varphi $ with $\sup \varphi \geq 0$
(resp. with $\varphi (P_{i})\geq 0$ for $1\leq i\leq n$).
Suppose in addition that the orbit of each $P_{i}$ under the action of $G$
has positive measure.
Then $\alpha (\omega )
=\alpha _{_{\Delta _{n}}}(\omega )$.
\end{theorem}
We first establish a few lemmas which will be useful for the
proof.
\begin{lemma}\label{SecDeuxUnLemUn}
Let $\bigl(\varphi _{n}\bigr)_{n\geq 0}$ be a sequence of
admissible functions with nonnegative maxima. Then
there exists a subset $\Omega $ of $M$, with
$\mum (\Omega )=\mum(M)$, and a
subsequence $\varphi _{n_{k}}$ of $\varphi _{n}$, such that for
every $p
$ in $\Omega $, the sequence $\bigl(\varphi _{n_{k}}(p ) \bigr)_{k\geq 0}$
has a finite lower bound (depending on $p$).
\end{lemma}
\begin{proof}
It is sufficient to assume that $\varphi _{n}$ has null maxima.
Let $Q_{n}$ be a point such that $\varphi _{n}(Q_{n})$ vanishes.
Green's formula runs as follows:
$$
\varphi _{n}(Q_{n})=\dfrac {1}{V}\int _{M}\varphi _{n}+
\int _{M}G(Q_{n},R)\, \Delta \varphi _{n}(R)\, dV(R),
$$
with $G(Q,R)\geq 0$ and $\ds\int _{M}G(Q,R)\, dV(R)=C$, where
$C$ is a positive constant (see \cite{Aub8}).
Since $\varphi _{n}$
is admissible, $\Delta\varphi _{n}$ is less than $m$,
$m$ being the dimension of
$M$. Thus $\ds\int _{M}|\varphi _{n}|\leq C\, m\, V$.
Furthermore,
$\ds\int_{M}\Delta \varphi _{n}=0$, so
$\ds\int_{M}|\Delta \varphi _{n}|=2\ds\int_{\{\Delta \varphi _{n}>0\}}
\Delta \varphi _{n}\leq 2mV$. For every $Q$ in $M$, we have
$
\ds\nabla \varphi _{n}(Q)=\ds\int_{M}\nabla_{Q}G(Q,R)\Delta \varphi
_{n}(R)dv(R)
$, so that
$$
\int _{M}|\nabla\varphi _{n}|\leq
\int_{M}\Bigl[\int _{M}|\nabla_{Q}G(Q,R)|dv(Q)\Bigr] |\Delta \varphi _{n}(R)|dv(R)\leq
2m\widetilde{C}V,
$$
since $\ds\int _{M}|\nabla_{Q}G(Q,R)|dv(Q)$ is a continuous, hence a bounded
function on $M$.
Thus $(\varphi _{n})_{n\geq 0}$ is bounded in the Sobolev space
$H^{1,1}(M)$. By Kondrakov's theorem, we can extract from
$(\varphi _{n})_{n\geq 0}$ a subsequence which converges in $L^{1}(M)$,
and after an other extraction we can suppose that this sequence
converges almost everywhere
to a function $\varphi $ of $L^{1}(M)$. Since $\varphi $ is
finite almost everywhere, we get the result.
\end{proof}
\begin{lemma}\label{SecDeuxUnLemTrois}
Let $\bigl(\varphi _{n}\bigr)_{n\geq 0}$ be a sequence of
admissible functions with nonnegative maxima and suppose that
there exists a
compact group $G$ of holomorphic isometries of $M$ such that the orbit
of each $P_{i}$ has positive measure.
Let $\Phi :G\lga  \R\cup\{-\infty \}$ be the map defined by
$\Phi (g)=\ds\inf_{\Delta _{n}}\inf_{k\geq 0}(\varphi_{k}\circ g
)$. Then there exists $g$ in $G$ such that $\Phi (g)$ is finite.
\end{lemma}
\begin{proof}
Suppose that $\Phi \equiv -\infty $.
For $i=1,\dots ,n$,
let
$A_{i}$ be the set of the
$g$ in $G$ such that
$\ds\inf_{k\geq 0}(\varphi _{k}\circ g) (P_{i})
=-\infty $. The sets $A_{i}$ are measurable and $\ds\cup_{i=1}^{n}
A_{i}=G$, so there exists $i$ such that $A_{i}$ has positive measure.
From
Lemma (\ref{SecDeuxUnLemUn}), $A_{i}.P_{i}$ is  a subset of $ \Omega ^{c}$.
Since $\Omega $ and $M$ have the same measure, the measure of
$A_{i}.P_{i}$ vanishes. Let $u_{i}$ be the map from $G$ to $M$ which
sends $g$ to
$g(P_{i})$. Then $u_{i}$ has constant rank on $G$. Indeed,
$u_{i}\circ L(g)=\sigma _{g}\circ u_{i}$, where $L(g)$ is the left
translation by $g$ and $\sigma _{g}$ the map from $M$ to $M$ which sends
$x$ to $g.x$. Since $G.P_{i}$ has positive measure, $u_{i}$ is a
submersion on $G$, so that $u_{i}(A_{i})$ has positive measure. This is
a contradiction since $u_{i}(A_{i})=A_{i}.P_{i}$.
\end{proof}
We can now prove Theorem (\ref{SecDeuxUnProp}).
\begin{proof}
It is clear that $\alpha (\omega )
\leq \alpha _{_{\Delta _{n}}}(\omega )$. Conversely, let $\varepsilon
>0$. There exists a sequence $\bigl(\varphi _{n} \bigr)_{n\geq 0}$ of
admissible functions with positive maxima such that
$\ds\int _{M}e^{-(\alpha (\omega )+\varepsilon )\varphi _{k}}$
goes to infinity  as $k$ goes to infinity.
Replacing $\varphi _{n}$ by $\varphi _{n}-\sup \varphi _{n}$,
we can take $\sup \varphi _{n}=0$.
First we apply Lemma(\ref{SecDeuxUnLemUn}). For the sake of
simplicity, we take $\varphi _{n_{k}}=\varphi _{k}$.
From Lemma (\ref{SecDeuxUnLemTrois}), there exists an element $g$ in
$G$ such that $\Phi (g)
$ is finite; we define $\Psi _{k}$ by
$\Psi _{k}=\varphi _{k}\circ g-\Phi (g)$.
Since $g$ is an
isometry, $\Psi _{k}$ is $\omega $-admissible, and
from the very definition of $\Phi  $,
$\Psi _{k}(P_{i})$
is nonnegative. Furthermore,
$\ds
\int _{M}e^{-(\alpha (\omega )+\varepsilon  )\Psi _{k}}
=e^{(\alpha (\omega )+\varepsilon  )\Phi (g)}
\int _{M}e^{-(\alpha (\omega )+\varepsilon  )\varphi _{k}}
$. This proves that
$\ds\int _{M}e^{-(\alpha (\omega )+\varepsilon  )\Psi _{k}}
$ goes to infinity as $k$ goes to infinity.
Then, $\alpha _{_{\Delta _{n}}}(\omega )\leq
\alpha (\omega )+\varepsilon $.
This inequality holds for every positive $\varepsilon$,
and so $\alpha _{_{\Delta _{n}}}(\omega )\leq
\alpha (\omega )$.
\end{proof}
\subsection{Tian's invariant on a product}
For a K\"{a}hler form $\omega $ on a compact K\"{a}hler manifold $M$,
$\alpha (\omega )$ is defined as in Theorem (\ref{SecDeuxUnProp}).
\begin{proposition}\label{SecDeuxPropDeux}
Let $\bigl(M_{i}\bigr)_{1\leq i\leq n}$
be compact K\"{a}hler manifolds with metric forms $\bigl(
\omega _{i} \bigr)_{1\leq i\leq n}$. We endow the product $M_{1}\times\dots
\times M_{n}$ with the metric $\omega _{1}\oplus\dots \oplus \omega _{n}$.
Then
$\alpha (\omega _{1}\oplus\dots \oplus \omega _{n})=\ds\inf_{1\leq i\leq n}
\alpha (\omega _{i})$.
\end{proposition}
\begin{proof}
It suffices to make the proof when $n=2$, the general result will follow
by induction.
\par\medskip
(1) Suppose that $\alpha (\omega _{1})\leq \alpha (\omega _{2})$, and
let $\varepsilon >0$. There exists a sequence
$\bigl(\varphi _{n}\bigr)_{n\geq 0}$ of $\omega _{1}$-admis\-si\-ble
functions on $M_{1}$  with positive
maxima such that $\ds\int _{M_{1}}e^{-\bigl(
\alpha (\omega _{1})+\varepsilon \bigr)\varphi _{n}}$ goes to infinity
when $n$ goes to infinity. We
define $\psi _{n}$ on $M_{1}\times M_{2}$ by
$\psi _{n}(m_{1},m_{2})=\varphi _{n}(m_{1})$.
Thus $\psi _{n}$ is $(\omega _{1}
\oplus\omega _{2})$-admissible on $M_{1}\times M_{2}$, with positive
maximum, and
$
\ds\int _{M_{1}\times M_{2}}e^{-\bigl(\alpha (\omega _{1})+\varepsilon \bigr)
\psi _{n}}=V(M_{2})\int _{M_{1}}
e^{-\bigl(\alpha (\omega _{1})+\varepsilon \bigr)
\varphi _{n}}
$, so that $\ds\int _{M_{1}\times M_{2}}e^
{-\bigl(\alpha (\omega _{1})+\varepsilon \bigr)
\psi _{n}}$ goes to infinity
when $n$ goes to infinity. We have therefore $\alpha (\omega _{1}
\oplus\omega _{2})\leq \alpha (\omega _{1})+\varepsilon $. This yields
$\alpha (\omega _{1}
\oplus\omega _{2})\leq \alpha (\omega _{1}) $.
\par\medskip
(2) Let us now prove the opposite inequality. Let $\alpha $ be a real
number such that
$\alpha <\inf\bigl(\alpha (\omega _{1}),
\alpha (\omega _{2})\bigr)$ and $\varphi $ an
$(\omega _{1}\oplus\omega _{2})$-admissible function on $M_{1}\times M_{2}$.
If
$m_{2}
$ is in $ M_{2}$, the function which sends $m_{1}$ to $ \varphi (m_{1},m_{2})$
is $\omega _{1}$-admissible. The same
holds for $M_{1}$. Let $(u,v)$ in $ M_{1}\times M_{2}$ be such that
$\varphi (u,v)\geq 0$. Then
\begin{align*}
\int _{M_{1}\times M_{2}}e^{-\alpha  \varphi (m_{1},m_{2})}dV_{1}\,
dV_{2}&=\int _{M_{1}}e^{-\alpha  \varphi (m_{1},v)}
\biggl(\,
\int _{M_{2}}e^{-\alpha \bigl[\varphi (m_{1},m_{2})
-\varphi (m_{1},v)\bigr]}dV_{2}
\biggr)
dV_{1}\\
&\leq C_{2}\int _{M_{1}}e^{-\alpha\varphi (m_{1},v)}dV_{1}\leq
C_{1}\, C_{2}.
\end{align*}
Thus, $\alpha \leq \alpha \bigl(\omega _{1}\oplus\omega _{2}\bigr)$ and
we get $\inf\bigl(\alpha (\omega _{1}),\alpha (\omega _{2})\bigr)
\leq \alpha \bigl(\omega _{1}\oplus\omega _{2}\bigr)$.
\end{proof}
\subsection{Tian's invariant on $\gpqc$}
Since there is a natural duality isomorphism between
$\gpqc$ and $G_{q,p}(\C \, )$, we can assume that $p\leq q$ without
loss of generality.
\subsubsection{Imbedding of $\bigl\{\prc{1}\bigr\}^{p}$
into $\gpqc$ when $p\leq q$}
For $w$ in $\C^{\, p(q-1)}$,
$w=\bigl(w_{i,j}\bigr)_{_{{
1\leq i\leq q\atop 1\leq j\leq p}\atop i\not = j}}$, we
define the map $\ti\rho_{w}$ from $ \Bigl\{\!\
\C^{2}\setminus(0,0)\!\Bigr\}^{p}$ to $ \ma{p+q}{p}$
by $$\ti\rho _{w}\Bigl((\lambda _{i},\mu _{i})_{1\leq i\leq p}\Bigr)
=\left\{\!\!
\begin{array}{lcl}
\lambda _{i}\, \delta _{ij}&\ \mbox{if}\ &i\leq p\\
w_{i-p,j}\, \lambda _{j}&\ \mbox{if}\ &i>p\ \mbox{and}\ i\not = j+p\\
\mu _{i}&\ \mbox{if}\ &i>p\ \mbox{and}\ i = j+p
\end{array}
\right.$$
We make, for $p+1\leq i\leq p+q $, the following row transformations:
$L_{i}\leftarrow L_{i}-\ds\sum _{1\leq j\leq p\atop i\not = j+p}
w_{i-p,j}\, L_{j}$. We get a matrix $\bigl(c_{i j}\bigr)_{_{
1\leq i\leq p+q\atop 1\leq j\leq p}}$ with $c_{i j}=\delta _{i j}\,
\lambda _{i}$ if $1\leq i\leq p$ and $c_{i j}=\delta _{i-p,j}\, {\mu _{j}}$
if $p+1\leq i\leq p+q$, which has rank $p$. $\ti\rho _{w}$ induces a map
from $\bigl\{\prc{1}\bigr\}^{p}$ into $\gpqc$ as shown on the following
diagram, where $\gamma $ is the projection of the principal fiber bundle
$\Bigl\{\!\C^{2}\setminus(0,0)\!\Bigr\}^{p}$ onto
$\bigl\{\prc{1}\bigr\}^{p}$. Remark that
$\ti\rho _{w}$ sends $[0,1]\times\dots \times [0,1]
$ onto $\pi (A)$, where
$m_{\{p+1,\dots ,2p\}}(A)=I^{(p)}$ and
$m_{\{p+1,\dots ,2p\}^{c}}(A)=0^{(q\times
p)}$.
\par\bigskip
\hspace*{5 cm}\xymatrix{
\bigl\{\C^{2}\setminus (0,0)\bigr\}^{p}\ar[r]^{\ti\rho _{w}}
\ar[d]_{\gamma  }&\mpq\ar[d]^{\pi }\\
\bigl\{\prc{1}\bigr\}^{p}\ar
[r]_{\rho _{w}}&
\gpqc
}
\par
We have
\begin{align*}
\bigl(\pi \circ\ti\rho _{w}\bigr)^{*}\bigl(\gpq\bigr)&=
i\, \dd\, \log \Bigl(\det\bigl(^{t}\!\ti\rho _{w}\,
\ba{\ti\rho }_{w}\bigr)\Bigr)\\
&=i\, \dd\, \log\Biggl( \dfrac {\det\,
\bigl(^{t}\!\ti\rho _{w}\, \ba{\ti\rho }_{w}\bigr)}
{\ds\prod_{k=1}^{p}\bigl(|\lambda _{k}|^{2}+|\mu _{k}|^{2}\bigl)}\Biggr)
+\sum _{k=1}^{p}i\, \dd\, \log \bigl(
|\lambda _{k}|^{2}+|\mu _{k}|^{2}
\bigr)\\
&=i\, \dd\, \log \ti{\Phi }+\gamma ^{*}\bigl(FS_{1}\oplus\dots \oplus
FS_{1}\bigr),
\end{align*}
where $FS_{1}$ is the Fubini-Study metric on $\prc{1}$.
$\ti{\Phi }$ is
invariant by the action of the structural group $\C^{*}\times\dots \times
\C^{*}$, so it induces a map $\Phi $ from $\bigl\{\prc{1}\bigr\}^{p}$ into
$\C$. Note that $\Phi \bigl([0,1]\times\dots \times [0,1]\bigr)=1$. Then
$\bigl(\pi \circ\ti\rho _{w}\bigr)^{*}\bigl(\gpq\bigr)=
\pi ^{*}\bigl(i\, \dd\, \log \Phi +FS_{1}\oplus\dots \oplus
FS_{1}\bigr)$, so that
$\rho_{w} ^{*}\bigl(\gpq\bigr)=
i\, \dd\, \log \Phi +FS_{1}\oplus\dots \oplus
FS_{1}$.
\subsubsection{Lower bound of $\alpha (\gpq)$}
For $I$ in $\mathcal{I}$, we define $P_{I}$ by $\mm{I}(P_{I})=
I^{(p)}$ and $\mm{{I^{c}}}(P_{I})=0^{(q\times p)}$. If
$n={p+q\choose p}$, we set $\Delta _{n}=\bigl\{P_{I}\bigr\}
_{I\in \mathcal{I}}$. Since $\upq$ is a transitive group of
holomorphic isometries
of $\gpqc$, we know from proposition (\ref{SecDeuxUnProp}), that
$\alpha \bigl(\gpq\bigr)=\alpha _{_{\Delta _{n}}}\bigl(\gpq\bigr)$.
We set $I=\{p+1,\dots ,2p\}$. Let $\varphi $ be an admissible function
on $\gpqc$, nonnegative on $\Delta _{n}$. The last equality of
the precedent section shows that the function
$\varphi \circ\rho _{w}+\log \Phi $ is
$\bigl(FS_{1}\oplus\dots
\oplus FS_{1}\bigr)$-admissible for every $w$ in $\C^{\, p(q-1)}$.
Furthermore,
$\bigl(\varphi \circ\rho _{w}+\log \Phi \bigr)$ sends $
[0,1]\times\dots \times [0,1]$ to
the nonnegative number $\varphi (P_{_I})$. It is
known that $\alpha (FS_{1})=1$ (see \cite{Aub6}). Proposition (\ref
{SecDeuxPropDeux}) yields $\alpha \bigl(FS_{1}\oplus\dots \oplus
FS_{1}\bigr)=1$.
\par
Let $\alpha $ be a real number such that
$\alpha <1$. There exists a constant $C$,
independent of $\varphi $, such that
$\ds\int _{{\bigl\{\prc{1}\bigr\}^{p}}}e^{-\alpha \varphi \circ\rho _{w}}\, \Phi
^{-\alpha }
\leq C$.
We define the map $F_{I}$ from $\pi (U_{I})$ to $\R_{+}$ by
$F_{I}(Z_{I})=\det\, \bigl(Id+^{t}\!Z_{I}\, \ba{Z}_{I}\bigr)$. On
$\bigl\{\prc{1}\bigr\}^{p}$, we work with the coordinates
$\mu _{1},\dots ,\mu_{p}$ in the chart $\lambda _{1}=\dots =\lambda
_{p}=1$. Thus
$$
\Phi (\mu )=\dfrac {F_{I}\circ \rho _{w}(\mu )}
{\ds\prod_{k=1}^{p}\bigl(1+|\mu _{k}|^{2}\bigr)},\quad \mbox{so that}\quad
\ds\int _{\mu \in \C^{\, p}}
e^{-\alpha  \varphi \circ\rho _{w}(\mu)}\,
\dfrac {dV_{\mu }\bigl(\C^{\, p}\bigr)}
{\ds\prod_{k=1}^{p}\bigl(1+|\mu _{k}|^{2}\bigr)^{2-\alpha }\,
\bigl(F_{_I}\circ\rho _{w}(\mu )\bigr)^{\alpha }}\leq C.
$$
We have the inequality $\ds\sum _{i=1}^{q}
\sum _{j=1}^{p}\bigl|Z_{i j}\bigr|^{2}\leq F_{I}\bigl(P_{I}\bigr)$.
In particular, for every $k$ in $\{1,\dots ,p\}$,
$1+|\mu _{k}|^{2}\leq F_{I}\circ\rho _{w}(\mu )$, and
$f_{I}\circ\rho _{w}(\mu )\geq
1+\ds\sum _{{1\leq i\leq q\atop 1\leq j\leq p}\atop
i\not = j}\bigl |w_{i j}\bigr |^{2}$.
Thus,
for $\kappa >0$ and $w\in\C^{\, p(q-1)}$,
$$
\dfrac {\ds\prod_{k=1}^{p}\bigl(1+|\mu _{k}|^{2}\bigr)^{2-\alpha }}
{\bigl(F_{I}\circ\rho _{w}(\mu )\bigr)^{\kappa +p+q-\alpha }}\leq
\dfrac {1}{\bigl(F_{_I}\circ\rho _{w}(\mu )\bigr)^{\kappa -p+q+\alpha (p-1)}}
\leq \dfrac {1}{\Bigl(1+\ds\sum _{{1\leq i\leq q\atop 1\leq j\leq p}\atop
i\not = j}\bigl |w_{i j}\bigr |^{2}\Bigr)^{\kappa }}=\dfrac {1}{\Bigl(
1+\bigl\|w\bigr\|^{2}\Bigr)^{\kappa }}\cdot
$$
We have, according to Proposition (\ref{SecUnTheoUn}),
\begin{align*}
&\int _{\pi (U_{I})}
\dfrac {e^{-\alpha \varphi }}{F_{I}^{\kappa }}=\int _{w\in\C^{\, p(q-1)}}
\int _{\mu \in\C^{\, p}}
\dfrac {e^{-\alpha  \varphi \circ\rho _{w}(\mu )}}
{\bigl(F_{I}\circ\rho _{w}(\mu )\bigr)^{\kappa +p+q}}\,
dV_{\mu }\bigl(\C^{\, p}\bigr)
\, dV_{w}\bigl(\C^{\, p(q-1)}\bigr)\\[2 ex]
&\quad =\int _{w\in\C^{\, p(q-1)}}\int _{\mu \in\C^{\, p}}
\Biggl(\dfrac {e^{-\alpha  \varphi \circ\rho _{w}(\mu )}}
{\ds\prod_{k=1}^{p}\bigl(1+|\mu _{k}|^{2}\bigr)^{2-\alpha }\,
\bigl(F_{I}\circ\rho _{w}(\mu )\bigr)^{\alpha }}\Biggr)\\[-2.4 ex]
&\hspace*{5 cm}\times\dfrac {\ds\prod_{k=1}^{p}
\bigl(1+|\mu _{k}|^{2}\bigr)^{2-\alpha }}
{\bigl(F_{I}\circ\rho _{w}(\mu )\bigr)^{\kappa +p+q-\alpha }}
\, dV_{\mu }\bigl(\C^{\, p}\bigr)
\, dV_{w}\bigl(\C^{\, p(q-1)}\bigr)\\[2ex]
&\quad =\int _{w\in\C^{\, p(q-1)}}\Biggl(
\int _{\mu \in\C^{\, p}}
\dfrac {e^{-\alpha  \varphi \circ\rho _{w}(\mu )}}
{\ds\prod_{k=1}^{p}\bigl(1+|\mu _{k}|^{2}\bigr)^{2-\alpha }\,
\bigl(F_{_I}\circ\rho _{w}(\mu )\bigr)^{\alpha }}
\, dV_{\mu }\bigl(\C^{\, p}\bigr)
\Biggr)\, \times \dfrac {dV_{w}\bigl(\C^{\, p(q-1)}\bigr)}
{\bigl(1+\bigl\|w\bigr\|^{2}\bigr)^{\kappa }}\\[2ex]
&\quad \leq C\int  _{w\in\C^{\, p(q-1)}}
\dfrac {dV_{w}\bigl(\C^{\, p(q-1)}\bigr)}
{\bigl(1+\bigl \|w\bigr \|^{2}\bigr)^{\kappa }}
\leq C' \qquad \mbox{if}\ \kappa >p(q-1).
\end{align*}
Thus, we obtain that for all $I$ in $\mathcal{I}$,
$\ds\int_{\pi (U_{I})}
\dfrac {e^{-\alpha  \varphi }}{F_{I}^{\kappa }}\leq C $, where $C$ is
independent of $\varphi $.
\par
Since $\gpqc$ is compact, there exists a
family $\bigl(V_{I}\bigr)_{I\in \mathcal{I}}$ of open sets of $\gpqc$
such that $V_{I}$ is relatively compact in $ \pi (U_{I}) $ for every
$I\in\mathcal{I}$, and $\bigcup_{I\in\,\mathcal{I}}V_{I}=\gpqc$. There
exists $M>0$ such that $F_{I}\leq M$ on
$V_{I}$ for every $I\in\mathcal{I}$. Thus
$$
\int _{\gpqc}e^{-\alpha \varphi }\leq \sum _{I\in\,\mathcal{I}}
\int _{V_{I}}e^{-\alpha \varphi }\leq
\sum _{I\in\,\mathcal{I}}M^{\kappa }
\int _{V_{I}}\dfrac {e^{-\alpha \varphi }}
{F_{I}^{\kappa }}\\
\leq M^{\kappa } \sum _{I\in\,\mathcal{I}}
\int _{\pi (U_{I})}\dfrac {e^{-\alpha \varphi }}
{F_{I}^{\kappa }}\leq C\, M^{\kappa }\, {p+q\choose p}.
$$
We deduce that $\alpha \bigl(\gpq\bigr)\geq 1$.
\subsubsection{Upper bound of $\alpha (\gpq)$}
We use here a method which can be found in \cite{Re} for the
complex projective space.
Let $I$ in $ \mathcal{I}$. We define $\ti{K}$ from  $\mpq$
to $\prr{1}$ by the
relation
$\ti{K}(M)=\bigl[\, |\det\, \mm{_I}(H)\, |^{2},\det \, ^{t}\!M\ba{M}\, \bigr]$.
$\ti{K}$ is invariant by the action of the structural group $G_{p}(\C\, )$,
so it induces a $C^{\infty }$ map $K$ from $\gpqc$ to $\prr{1}$.
Remark that $\psi =\log K$ is a K\"ahler potential on $U_{I}$ for the
metric $\gpq$.
\begin{lemma}\label{SecTroisLemmeUn}
There exists a decreasing sequence $\bigl({\varphi _{n}}\bigr)_{n\geq 0}$ of
admissible functions with positive maxima which converges pointwise to $-\psi $
on $\pi \bigl(U_{I}\bigr)$.
\end{lemma}
\begin{proof}
We construct a decreasing sequence $\bigl(f_{n}\bigr)_{n\geq 0}$ of $C^{\infty
}$ convex functions on $\R_{+}$ satisfying the conditions $1+f'_{n}>0$,
$f_{n}(x)=-\bigl(1-1/n\bigr)x$ for $x$ in $[0,n]$ and $f_{n}(x)=-n $ for
$x\geq 2n$.
Let $y$ be an element of  $\pi (U_{I})^{c}$
and $\Omega _{n}$ the set of the elements $x$ in $\pi (U_{I})$ such that
$\psi (x)>2n$.
Since
$F_{I}(y)=[0,1]$, there exists a neighborhood $V$ of y such that the
inequality $z>e^{2n}$ holds for every point
$[1,z]$ in $ F_{I}(V)$. Thus $V\cap\pi \bigl(U_{I}\bigr)
$ is included in $\Omega _{n}$. We have proved that
$W_{n}=\Omega _{n}\cup\pi \bigl(U_{I}\bigr)^{c}$,
so that $W_{n}$ is an open neighborhood
of $\pi \bigl(U_{I}\bigr)^{c}$. We define
$\varphi _{n}$ by $\varphi _{n}=f_{n}\circ \psi $ on $\pi \bigl(U_{I}\bigr)$
and $\varphi _{n}=-n$ on $W_{n}$.
Thus $\varphi _{n}$ is well defined and $\varphi _{n}(0)=0$. It remains
to show that $\varphi _{n}$ is admissible on $\pi
\bigl(U_{I}\bigr)$. We have
$$
\Bigl(\gpq+i\, \dd\, \varphi _{n}\Bigr)_{\lambda \,  \ba{\mu }}
=\partial _{\lambda \,  \ba{\mu }}\psi +
\partial_{\lambda } \bigl(f'_{n}\circ \psi \bigr)\partial_{\ba{\mu }}
\psi  =
\bigl(1+f'_{n}\circ \psi \bigr)\partial_{\lambda \,  \ba{\mu }}
\psi +f'' _{n}\circ \psi  \, \partial_{\lambda  }\psi \,
\partial _{\, \ba{\mu }}\psi .
$$
Hence the matrix of the metric $\gpq+i\, \dd\, \varphi _{n}$ is
of the form $A+T$ where $A$ is positive definite and $T$
has rank one and positive trace. So $A+T$ is positive definite and we get
the result.
\end{proof}
\begin{lemma}\label{SecTroisLemmeDeux}
Let $n$ in  $\N^{*}$ and $r$ a positive real number. Then
$$\ds\int _{||X||\leq r}
\dfrac {dV_{X}(M_{n}(\C))}
{\bigl |\det X\bigr |^{2 }}=+\infty .$$
\end{lemma}
\begin{proof}
We can write
$$
\ds\int _{||X||\leq r}
\dfrac {dV_{X}(M_{n}(\C))}
{\bigl |\det X\bigr |^{2 }}=\sum _{k=0}^{\infty }
\ds\, \, \, \int _{{r}/{2^{k+1}}
\leq ||X|| \leq r/2^{k}}
\dfrac {dV_{X}(M_{n}(\C))}
{\bigl |\det X\bigr |^{2 }}\cdot
$$
We put $Y=2^{k}X$, so
$$
\int _{
r/2^{k+1}\leq ||X|| \leq r/2^{k}}
\dfrac {dV_{X}(M_{n}(\C))}
{\bigl |\det X\bigr |^{2 }}=
\int _{
1/2\leq ||Y|| \leq 1}
\dfrac {dV_{Y}(M_{n}(\C))}
{\bigl |\det Y\bigr |^{2 }}\cdot
$$
The terms in the series are strictly positive and independent of $k$.
The sum is therefore infinite.
\end{proof}
We can now prove that
$\alpha \bigl(\gpq\bigr)$ is upper bounded by $1$.
Suppose that $\alpha \bigl(\gpq\bigr)>1$. Then there exists a positive $C$
such that for every integer $n$,
$\ds\int _{\pi (U_{I})}e^{-\varphi _{n}}\leq C$.
Using Lemma (\ref{SecTroisLemmeUn}) and monotonous convergence,
$\ds\int _{\pi (U_{I})}F_{I}\leq C$.
Since
$\pi(U_{I})^{c}$ has zero measure, $\ds\int _{\gpqc}F_{I}\leq C$. Let
$\ti{I}$ in $\mathcal{I}$ be such that $I\cap\ti{I}=\emptyset$ (this is possible
since $p\leq q$). We have $P_{ {\ti{I}}}\bigl\{\mm{I}\,
\bigl(P_{{\ti{I}} }\bigr)\bigr\}^{-1}=P_{I}$. Remark that
$\mm{I}\bigl(P_{{\ti{I}}}\bigr)=
\mm{I}\bigl(Z_{{\ti{I}}}\bigr)$. Thus
$\det\bigl(Id+^{t}\!Z_{I}\, \ba{Z}_{I}\bigr)=\det\bigl(^{t}\!
P_{\ti{I}}\, \ba{P}_{{\ti{I}}}\bigr)\,
\bigl |\det\mm{I}\bigl(Z_{{\ti{I}}}\bigr)\bigr |^{-2}$.
For $\bigl \|Z_{{\ti{I}}}\bigr \|\leq r$,
$\det\bigl(^{t}\!
P_{\ti{I}}\, \ba{P}_{{\ti{I}}}\bigr)\leq M$, so that
$
\ds\int_
{||Z_{\ti{I}}||\leq r}
\dfrac {dV_{Z_{\ti{I}}}(\ma{q}{p})}
{\bigl |\det \mm{I}\bigl(Z_{
{\ti{I}}}\bigr)\bigr |^{2 }}<+\infty
$. Integrating over the remaining variables
$\bigl(Z_{ij}\bigr)_{i\in\widetilde{I}^{c}\cap I^{c}}$
yields
$
\ds\int _{||Z||\leq r}
\dfrac {dV_{Z}(M_{p}(\C))}
{\bigl |\det Z\bigr |^{2 }}<+\infty
$, which is in contradiction with the result of Lemma
(\ref{SecTroisLemmeDeux}).
Thus we obtain $\alpha \bigl(\gpq\bigr)\leq 1$.

\end{document}